\documentclass[11pt, twoside, a4paper]{article}

\usepackage{amsmath}     				
\usepackage{amssymb}   					
\usepackage{amsthm}     				
\usepackage{graphicx}    				

\usepackage{textcomp}    				
\usepackage[T1]{fontenc} 				
\usepackage{marvosym}    				
\usepackage[sc]{mathpazo}  	 			

\usepackage{authblk} 					
\usepackage[usenames]{xcolor}  			
\usepackage{lastpage} 					

\usepackage{enumerate}	 				
\usepackage{mathtools}
\usepackage[toc]{appendix}
\mathtoolsset{showonlyrefs}
\usepackage{eufrak,mathrsfs}
\usepackage{soul,xfrac}			
\usepackage[normalem]{ulem}	
\usepackage[noadjust]{cite}			

\usepackage{hyperref} 					

\definecolor{ForestGreen}{rgb}{0.15,0.416,0.18}
\definecolor{EgyptBlue}{rgb}{0.063,0.2,0.65}
\hypersetup{
	colorlinks=true,
	linkcolor=EgyptBlue,         
   citecolor=EgyptBlue,          
   urlcolor=ForestGreen          
}

\newtheorem{theorem}{Theorem}[section]
\newtheorem{corollary}[theorem]{Corollary}
\newtheorem{lemma}[theorem]{Lemma}
\newtheorem{proposition}[theorem]{Proposition}

\newtheorem{asum}[theorem]{Assumption}
\newtheorem{assum}[theorem]{Assumptions}

\theoremstyle{definition}
\newtheorem{definition}[theorem]{Definition}
\theoremstyle{definition}
\newtheorem{remark}[theorem]{Remark}
\theoremstyle{definition}

\numberwithin{equation}{section}
\numberwithin{table}{section}
\numberwithin{figure}{section}

\title{On the stochastic Allen-Cahn equation on networks with multiplicative noise}

\author[1, 2]{\textbf{Mih\'aly Kov\'acs}}
\author[3, 4]{\textbf{Eszter Sikolya}}

\affil[1]{Faculty of Information Technology and Bionics, P\'azm\'any P\'eter Catholic University, Pr\'ater u. 50/A., Budapest, H--1083, Hungary}
\affil[2]{Chalmers University of Technology and University of Gothenburg, SE-412 96 Gothenburg, Sweden}
\affil[3]{Institute of Mathematics, E\"otv\"os Lor\'and University, P\'azm\'any P\'eter stny. 1/c.
Budapest, H--1117, Hungary}
\affil[4]{Alfr\'ed R\'enyi Institute of Mathematics, Re\'altanoda street 13–15, Budapest, H-1053, Hungary}

\newcommand{\doi}[1]{\url{https://doi.org/#1}}
\newcommand{\MR}[1]{\href{https://www.ams.org/mathscinet-getitem?mr=#1}{MR#1}}

\makeatletter
\renewcommand{\maketitle}{\bgroup\setlength{\parindent}{0pt}

\vspace{1truecm}
\begin{center}{\vbox{\titlefont\@title}}\end{center}
\vspace{0.5truecm}
\begin{center}{\@author} \end{center}

\egroup
}

\renewcommand{\@fnsymbol}[1]{%
    \ifcase#1 \or {\,\Letter\!} \or\textasteriskcentered\or \textasteriskcentered\textasteriskcentered 
    \else\@ctrerr\fi}
\makeatother

\newcommand*{\titlefont}{\fontsize{18}{21.6}\selectfont\textbf}

\makeatletter
\renewcommand\@author{\ifx\AB@affillist\AB@empty\AB@author\else
      \ifnum\value{affil}>\value{Maxaffil}\def\rlap##1{##1}%
    \AB@authlist\\[\affilsep]\vbox{\AB@affillist}
    \else  \AB@authors\fi\fi}
\makeatother

\def\e{\mathrm{e}}
\def\ve{\varepsilon}
\def\la{\lambda}
\def\me{\mathsf{e}}
\def\mv{\mathsf{v}}

\def\ea{\EuFrak{a}}

\def\ve{\varepsilon}

\def\real{\mathbb{R}}

\def\mcB{\mathcal{B}}

\def\mcW{\mathcal{W}}

\def\dt{\, dt}
\def\ds{\, ds}

\begin{document}

\maketitle

\pagestyle{plain}

\begin{center}
\noindent
\begin{minipage}{0.85\textwidth}\parindent=15.5pt
%
%

{\small{
\noindent {\bf Abstract.} We consider a system of stochastic Allen-Cahn equations on a finite network represented by a finite graph. On each edge in the graph a multiplicative Gaussian noise driven stochastic Allen-Cahn equation is given with possibly different potential barrier heights supplemented by a continuity condition and a Kirchhoff-type law in the vertices. Using the semigroup approach for stochastic evolution equations in Banach spaces we obtain existence and uniqueness of solutions with sample paths in the space of continuous functions on the graph. We also prove more precise space-time regularity of the solution.}
\smallskip

\smallskip

}

\end{minipage}
\end{center}

\section{Introduction}

We consider a finite connected network, represented by a
finite graph $G$ with $m$ edges $\me _1,\dots,\me _m$ and $n$
vertices $\mv_1,\dots,\mv_n$. 
We normalize and parametrize the edges on the interval $[0,1]$. We denote by $\Gamma(\mv_i)$ the set of all the indices of the edges
having an endpoint at $\mv _i$, i.e.,
\[\Gamma(\mv _i)\coloneqq\left\{j\in \{1,\ldots,m\}: \me _j(0)=\mv _i\hbox{ or } \me _j(1)=\mv _i\right\}.\]
Denoting by $\Phi\coloneqq(\phi_{ij})_{n\times m}$
 the so-called incidence matrix of the graph $G$, see Subsection \ref{subsec:systeq} for more details, we aim to analyse the existence, uniqueness and  regularity of solutions of the problem
\begin{equation}\label{eq:stochnet}
\left\{\begin{aligned}
\dot{u}_j(t,x) & =  (c_j u_j')'(t,x)-p_j(x)u_j(x,t)&&\\
 &\quad  +\beta_j^2u(x,t)- u(x,t)^3&&\\
 &\quad + g_j(t,x,u_j(t,x))\frac{\partial w_j}{\partial t}(t,x), &&t\in(0,T],\; x\in(0,1),\; j=1,\dots,m, \\
 u_j(t,\mv _i)& =  u_\ell (t,\mv _i)\eqqcolon q_i(t), &&t\in(0,T],\; \forall j,\ell\in \Gamma(\mv _i),\; i=1,\ldots,n,\\
\left[M q(t)\right]_{i} & =  -\sum_{j=1}^m \phi_{ij}\mu_{j} c_j(\mv_i) u'_j(t,\mv_i), && t\in(0,T],\; i=1,\ldots,n,\\
 u_j(0,x) & =  \mathsf{u}_{j}(x), &&x\in [0,1],\; j=1,\dots,m, 
\end{aligned}
\right.
\end{equation}
where $\frac{\partial w_j}{\partial t}$ are independent space-time white noises. The reaction terms in \eqref{eq:stochnet} are classical  Allen-Cahn nonlinearities $h_j(\eta)=-\eta^3+\beta_j^2 \eta$ with $\beta_j>0$, $j=1,\dots, m$. Note that $h_j=-H_j'$ where $H_j(\eta)=\frac14(\eta^2-\beta_j^2)^2$ is a double well potential for each $j$ with potential barrier height ${\beta_j^4}/{4}$. The diffusion coefficients $g_j$ are assumed to be locally Lipschitz continuous and of linear growth. 
The coefficients of the linear operator satisfy standard smoothness assumptions, see Subsection \ref{subsec:systeq}, the matrix $M$ satisfies Assumptions \ref{as:M} and $\mu_j$, $j=1,\dots, m$, are positive constants. The classical Allen-Cahn equation belongs to the class of phase field models and is a classical tool to model processes involving thin interface layers between almost homogeneous regions, see \cite{AC79}. It is a particular case of a reaction-diffusion equation of
bistable type and it can be used to study front propagations as in \cite{BBS92} . Effects due to, for example,
thermal fluctuations of the system can be accounted for by adding a Wiener type noise in the equation, see \cite{Co70}. 

While deterministic evolution equations on networks are well studied, see,  \cite{Al84,Al94,BFN16,BFN16b,Be85,Be88,Be88b,BN96,Ca97,CF03,EK19,Ka66,KMS07,KS05,Lu80,MS07,Mu07,MR07,Mu14} which is, admittedly, a rather incomplete list, the study of their stochastic counterparts is surprisingly scarce despite their strong link to applications, see e.g. \cite{BMZ08,BM10,Tu06} and the references therein. In \cite{BMZ08} additive L\'evy noise is considered that is square integrable with drift being a cubic polynomial. In  \cite{BZ14} multiplicative square integrable L\'evy noise is considered but with globally Lipschitz drifts $f_j$ and diffusion coefficients and with a small time dependent perturbation of the linear operator. Paper \cite{BM10} treats the case when the noise is an additive fractional Brownian motion and the drift is zero. In \cite{CP17a} multiplicative Wiener perturbation is considered both on the edges and vertices with globally Lipschitz diffusion coefficient and zero drift and time-delayed boundary condition.  Finally, in \cite{CP17}, the case of multiplicative Wiener noise is treated with bounded and globally Lipschitz continuous drift and diffusion coefficients and noise both on the edges and vertices. 

In all these papers the semigroup approach is utilized in a Hilbert space setting and the only work that treats non-globally Lipschitz continuous drifts on the edges, similar to the ones considered here,  is \cite{BMZ08} but the noise is there additive and square-integrable. In this case, energy arguments are possible using the additive nature of the equation which does not carry over to the multiplicative case. Therefore, we use an entirely different tool set based on the semigroup approach for stochastic evolution equations in Banach spaces \cite{vNVW08}, or for the classical stochastic reaction-diffusion setting \cite{KvN12,KvN19}, see also, \cite{BG99,BP99,Ce03,Pe95}. We are able to rewrite \eqref{eq:stochnet} in a form that fits into this framework. After establishing various embedding and isomorphy results of function spaces and  interpolation spaces,  we may use \cite[Theorem 4.9]{KvN19} to prove our main existence and uniqueness result, Theorem \ref{thm:SAC}, which guarantees  existence and uniqueness of solutions with sample paths in the space of continuous functions on the graph, denoted by $B$ in the paper (see Definition \ref{def:b}); that is,  in the space of continuous functions that are continuous on the edges and also across the vertices. When the initial data is sufficiently regular, then Theorem \ref{thm:SAC} also yields certain space-time regularity of the solution.

The paper is organized as follows. In Section \ref{sec:determnetwork} we collect partially known semigroup results for the linear deterministic version of \eqref{eq:stochnet}.  In Subsection \ref{subsec:KVN} we first recall an abstract result from \cite{KvN12,KvN19} regarding abstract stochastic Cauchy problems in Banach spaces. In order to utilize the abstract framework in our setting we prove various preparatory results  in Subsection \ref{subsec:prep}: embedding and isometry results are contained in Lemma \ref{lem:ApmaxW0G}, Lemma \ref{lem:Biso} and Corollary \ref{cor:fractionalspaceincl}, and a semigroup result in Proposition \ref{prop:mcAonC}. Subsection \ref{subsec:main} contains our main results where we first consider the abstract stochastic It\^o equation corresponding to a slightly more general version of \eqref{eq:stochnet}. An existence and uniqueness result for the abstract stochastic It\^o problem is contained in Theorem \ref{theo:SCPnsolcont} followed by a space-time regularity result  in Theorem \ref{theo:Holderreg}. These are then applied to the It\^o equation corresponding \eqref{eq:stochnet} to yield the main result of the paper, Theorem \ref{thm:SAC}, concerning the existence, uniqueness and space-time regularity of the solution of  \eqref{eq:stochnet}.

\section{Heat equation on a network}\label{sec:determnetwork}

\subsection{The system of equations}\label{subsec:systeq}

We consider a finite connected network, represented by a
finite graph $G$ with $m$ edges $\me _1,\dots,\me _m$ and $n$
vertices $\mv_1,\dots,\mv_n$. 
We normalize and parametrize the edges on the interval $[0,1]$. 

The structure of the network is given by the $n\times m$ matrices
$\Phi^+\coloneqq(\phi^+_{ij})$ and $\Phi^-\coloneqq(\phi^-_{ij})$ defined by
\begin{equation}\label{eq:fiijpm}
\phi^+_{ij}\coloneqq\left\{
\begin{array}{rl}
1, & \hbox{if } \me _j(0)=\mv _i,\\
0, & \hbox{otherwise},
\end{array}
\right.
\qquad\hbox{and}\qquad
\phi^-_{ij}\coloneqq\left\{
\begin{array}{rl}
1, & \hbox{if } \me _j(1)=\mv _i,\\
0, & \hbox{otherwise,}
\end{array}
\right.
\end{equation}
for $i=1,\ldots ,n$ and $j=1,\ldots m.$ We denote by $\me _j(0)$ and $\me _j(1)$ the $0$ and the $1$ endpoint of the edge $\me _j$, respectively.
We refer to~\cite{KS05} for terminology. The $n\times m$ matrix
$\Phi\coloneqq(\phi_{ij})$ defined by \[\Phi\coloneqq\Phi^+-\Phi^-\] is known in
graph theory as \emph{incidence matrix} of the graph $G$. Further,
let $\Gamma(\mv_i)$ be the set of all the indices of the edges
having an endpoint at $\mv _i$, i.e.,
\[\Gamma(\mv _i)\coloneqq\left\{j\in \{1,\ldots,m\}: \me _j(0)=\mv _i\hbox{ or } \me _j(1)=\mv _i\right\}.\]
For the sake of simplicity, we will denote the values of a continuous function defined on the (parameterized) edges of the graph, that is of
\[f=\left(f_1,\ldots ,f_m\right)^{\top}\in \left(C[0,1]\right)^m\cong C\left([0,1],\real^m\right)\]
at $0$ or $1$ by $f_j(\mv_i)$ if $\me _j(0)=\mv _i$ or $\me _j(1)=\mv _i$, respectively, and $f_j(\mv_i)\coloneqq0$ otherwise, for $j=1,\ldots ,m$.

We start with the problem
\begin{equation}\label{netcp}
\left\{\begin{array}{rclll}
\dot{u}_j(t,x)&=& (c_j u_j')'(t,x)-p_j(x)u_j(t,x), &t> 0,\; x\in(0,1),\; j=1,\dots,m, & (a)\\[0.1cm]
u_j(t,\mv _i)&=&u_\ell (t,\mv _i)\eqqcolon q_i(t), &t> 0,\; \forall j,\ell\in \Gamma(\mv _i),\; i=1,\ldots,n,& (b)\\[0.1cm]
\left[M q(t)\right]_{i} &=& -\sum_{j=1}^m \phi_{ij}\mu_{j} c_j(\mv_i) u'_j(t,\mv_i), &t> 0,\; i=1,\ldots,n,& (c)\\[0.1cm]
u_j(0,x)&=&\mathsf{u}_{j}(x), &x\in [0,1],\; j=1,\dots,m & (d)
\end{array}
\right.
\end{equation}
on the network. Note that $c_j(\cdot)$, $p_j(\cdot)$ and $u_j(t,\cdot)$ are
functions on the edge $\me_j$ of the network, so that the
right-hand side of~$(\ref{netcp}a)$ reads in fact as
\[(c_j u_j')'(t,\cdot)=\frac{\partial}{\partial x}\left( c_j\frac{\partial}{\partial x}u_j\right)(t,\cdot)-p_j(\cdot)u_j(t,\cdot), \qquad t\geq 0,\; j=1,\ldots,m.\]

The functions $c_1,\ldots,c_m$ are (variable) diffusion coefficients or conductances, and we assume that 
\[0<c_j\in C^1[0,1],\quad j=1,\ldots,m.\]

The functions $p_1,\ldots,p_m$ are nonnegative, continuous functions, hence
\begin{equation}\label{eq:pjpos}
0\leq p_j\in C[0,1],\quad j=1,\ldots,m.
\end{equation}

Equation~$(\ref{netcp}b)$ represents the continuity of the values attained by the system at the vertices in each time instant, and we denote by $q_i(t)$ the common functions values in the vertice $i$, for $i=1,\ldots,n$ and $t>0$.

In $(\ref{netcp}c)$, $M\coloneqq\left(b_{ij}\right)_{n\times n}$ is a matrix satisfying the following
\begin{asum}\label{asum:M1}
The matrix $M=\left(b_{ij}\right)_{n\times n}$ is real, symmetric and negative semidefinite, $M\not\equiv 0$.
\end{asum}

On the left-hand-side, $[Mq(t)]_{i}$ denotes the $i$th coordinate of the vector $M q(t)$. On the right-hand-side, the coefficients 
\[0<\mu_j,\quad  j=1,\ldots,m\] 
are strictly positive constants that influence the distribution of impulse happening in the ramification nodes according to the Kirchhoff-type law~$(\ref{netcp}c)$.

We now introduce the $n\times m$ \emph{weighted incidence matrices}
\[\Phi^+_w\coloneqq(\omega^+_{ij})\text{ and }\Phi^-_w\coloneqq(\omega^-_{ij})\] 
with entries
\begin{equation}\label{eq:matrixFiw}
\omega^+_{ij}\coloneqq\left\{
\begin{array}{ll}
\mu_j c_j(\mv_i), & \hbox{if } \me _j(0)=\mv _i,\\
0, & \hbox{otherwise},
\end{array}
\right. \qquad\hbox{and}\qquad \omega^-_{ij}\coloneqq\left\{
\begin{array}{ll}
\mu_j c_j(\mv_i), & \hbox{if } \me _j(1)=\mv _i,\\
0, & \hbox{otherwise}.
\end{array}
\right.
\end{equation}

With these notations, the Kirchhoff law $(\ref{netcp}c)$ becomes
\begin{equation}
\label{eq:Kir}
Mq(t)=-\Phi_w^+ u'(t,0)+\Phi_w^- u'(t,1), \qquad t\geq 0.
\end{equation}

In equation~$(\ref{netcp}d)$ we pose the initial conditions on the edges.

\subsection{Spaces and operators}

We are now in the position to rewrite our system in form of an abstract Cauchy problem, following the concept of \cite{KMS07}. First we consider the (real) Hilbert space
\begin{equation}\label{eq:E2}
E_2\coloneqq\prod_{j=1}^m L^2(0,1; \mu_j dx)
\end{equation}
as the \emph{state space} of the edges, endowed with the natural inner product
\[\langle u,v\rangle_{E_2}\coloneqq \sum_{j=1}^m \int_0^1 u_j(x)v_j(x) \mu_j dx,\qquad
u=\left(\begin{smallmatrix}u_1\\ \vdots\\
u_m\end{smallmatrix}\right),\;v=\left(\begin{smallmatrix}v_1\\ \vdots\\
v_m\end{smallmatrix}\right)\in E_2.\]

Observe that $E_2$ is isomorphic to $\left( L^2(0,1)\right)^m$ with equivalence of norms. 

We further need the \emph{boundary space} $\real^{n}$ of the vertices. According to $(\ref{netcp}b)$ we will consider such functions on the edges of the graph those values coincide in each vertex. Therefore we introduce the \emph{boundary value operator} 
\[L\colon\left(C[0,1]\right)^m\subset E_2\to \real^n\]
with
\begin{align}\label{eq:Ldef}
D(L) & = \left\{u\in \left(C[0,1]\right)^m: u_j(\mv _i)=u_\ell (\mv _i),\; \forall j,\ell\in \Gamma(\mv _i),\; i=1,\ldots,n\right\};\notag\\
L u & \coloneqq\left(q_1,\ldots ,q_n\right)^{\top}\in \real^n,\quad q_i=u_j(\mv _i)\text{ for some } j\in \Gamma(\mv _i),\; i=1,\ldots,n.
\end{align}
The condition $u(t,\cdot)\in D(L)$ for each $t>0$ means that $(\ref{netcp}b)$ is for the function $u(\cdot,\cdot)$ satisfied.

On $E_2$ we define the operator
\begin{equation}\label{eq:opAmax}
A_{max}\coloneqq\begin{pmatrix}
\frac{d}{dx}\left(c_1 \frac{d}{dx}\right)-p_1 & & 0\\
 & \ddots &\\
0 & & \frac{d}{dx}\left(c_m \frac{d}{dx}\right)-p_m \\
\end{pmatrix}
\end{equation}
with domain
\begin{equation}\label{eq:domAmax}
D(A_{max})\coloneqq\left(H^2(0,1)\right)^m\cap D(L).
\end{equation}

This operator can be regarded as \emph{maximal} since no other boundary condition except continuity is supposed for the functions in its domain.

We further define the so called \emph{feedback operator} acting on $D(A_{max})$ and having values in the boundary space $\real^{n}$ as
\begin{align}\label{eq:opC}
D(C) & = D(A_{max});\\
C u & \coloneqq-\Phi_w^+ u'(0)+\Phi_w^- u'(1),
\end{align}
compare with \eqref{eq:Kir}.

With these notations, we can finally rewrite \eqref{netcp} in form of an abstract Cauchy problem. Define
\begin{align}\label{eq:amain}
A&\coloneqq A_{max}\\
D(A)&\coloneqq\{u\in E_2:u\in D(A_{max})\text{ and } MLu=Cu\},
\end{align}
see the definitions above.
Using this, \eqref{netcp} becomes
\begin{equation}\label{eq:acp}
\left\{\begin{array}{rcll}
\dot{u}(t)&=& A u(t), &t> 0,\\
u(0)&=&\mathsf{u},
\end{array}
\right.
\end{equation}
with $\mathsf{u}=(\mathsf{u}_1,\dots ,\mathsf{u}_m)^{\top}$.

\subsection{Well-posedness of the abstract Cauchy problem}

To prove well-posedness of \eqref{eq:acp} we define a bilinear form on the Hilbert space $E_2$ with domain
\begin{equation}\label{eq:domform}
D\left(\ea\right)=V\coloneqq \left(H^1(0,1)\right)^m\cap D(L).
\end{equation}
as
\begin{equation}\label{eq:form}
\ea(u,v)\coloneqq\sum_{j=1}^m\int_0^1 \mu_j c_j(x) u'_j(x) v'_j(x) dx+\sum_{j=1}^m\int_0^1 \mu_j p_j(x)u_j(x)v_j(x) dx-\langle M q,r\rangle_{\real^n},
\end{equation}
where $Lu=q$ and $Lv=r$.
 
The next definition can be found e.g. in \cite[Section 1.2.3]{Ou05}.
\begin{definition}
From the form $\ea$ -- using the Riesz representation theorem -- we can obtain a unique operator
$\left(B,D(B)\right)$ in the following way:
\begin{align*}
D(B)&\coloneqq \left\{u\in V:\exists v\in E_2 \hbox{ s.t. } \ea(u,\varphi)=\langle v,\varphi\rangle_{E_2}\; \forall \varphi\in V\right\},\\
Bu&\coloneqq-v.
\end{align*}
We say that the operator $\left(B,D(B)\right)$ is \emph{associated with the form $\ea$}.
\end{definition}

In the following, we will claim that the operator associated with the form $\ea$ is $(A,D(A)).$ Furthermore, we will state results regarding how the properties of $\ea$ and the matrix $M$ carry on the properties of the operator $A$, obtaining the well-posedness of the abstract Cauchy-problem \eqref{eq:acp} on $E_2$ and even on $L^p$-spaces of the edges. The proofs of these statements combine techniques of \cite{Mu07} (where no $p_j$'s on the right-hand-side of $(\ref{netcp}b)$ are considered) and techniques of \cite{MR07} (where $p_j$'s are considered for the heat equation but the matrix $M$ is diagonal).

\begin{proposition}
The operator associated to the form $\ea$ \eqref{eq:domform}--\eqref{eq:form} is $(A,D(A))$ in \eqref{eq:amain}.
\end{proposition}
\begin{proof}
We can proceed similarly as in the proofs of \cite[Lemma 3.4]{Mu07} and \cite[Lemma 3.3]{MR07}.
\end{proof}

\begin{proposition}\label{prop:eaprop}
The form $\ea$ is densely defined, continuous, closed and accretive, hence $(A,D(A))$ is densely defined, dissipative and sectorial. Furthermore, $\ea$ is symmetric, hence the operator $(A,D(A))$ is self-adjoint. 
\end{proposition}
\begin{proof}
The first three properties of $\ea$ (densely defined, continuous and closed) follow analogous to the proof of \cite[Lemma 3.2]{MR07}. Since $M$ is dissipative (that is, negative semidefinite), and $p_j\geq 0$, $j=1,\dots ,m$, the form $\ea$ is accretive, see the proofs of \cite[Proposition 3.2]{Mu07} and \cite[Lemma 3.2]{MR07}. The symmetricity of $\ea$ follows from the fact that $M$ is real and symmetric, see the proof of \cite[Corollary 3.3]{Mu07}. The properties of $A$ follow now by \cite[Proposition 1.24, 1.51, Theorem 1.52]{Ou05}.
\end{proof}

As a corollary we obtain well-posedness of \eqref{eq:acp}.

\begin{proposition}\label{prop:determL2}
Assuming Assumption \ref{asum:M1} on the matrix $M$, the operator $\left(A, D(A)\right)$ defined in \eqref{eq:amain} generates a $C_0$ analytic, compact semigroup of contractions $\left(T_2(t)\right)_{t\geq 0}$ on $E_2.$ Hence, the abstract Cauchy problem \eqref{eq:acp} is well-posed on $E_2$.
\end{proposition}
\begin{proof}
The claim follows from Proposition \ref{prop:eaprop} and the fact that $\left(A, D(A)\right)$ is resolvent compact. This is true since $V$ is densely and compactly embedded in $E_2$ by the Rellich--Khondrakov Theorem, and we can use \cite[Theorem 1.2.1]{Da90}.
\end{proof}

In the following we will extend the semigroup $\left(T_2(t)\right)_{t\geq 0}$ on $L^p$-spaces. To this end we define

\begin{equation}\label{eq:Ep}
E_p\coloneqq\prod_{j=1}^m L^p(0,1; \mu_j dx),\quad p\in[1,\infty]
\end{equation}
and
\begin{equation}\label{eq:Epnorm}
\|u\|_{E_p}^p\coloneqq\sum_{j=1}^m \|u_j\|_{L^p(0,1; \mu_j dx)}^p,\quad u\in E_p,\quad p\in[1,\infty),
\end{equation}
\begin{equation}\label{eq:Einftynorm}
\|u\|_{E_{\infty}}\coloneqq\max_{j=1,\dots ,m} \|u_j\|_{L^{\infty}(0,1)},\quad u\in E_{\infty}.
\end{equation}

We can characterize features of the semigroup $\left(T_2(t)\right)_{t\geq 0}$ by those of $(\e^{tM})_{t\geq 0}$, the semigroup generated by the matrix $M$ -- hence, by properties of $M$. In particular, the following holds.

\begin{proposition}
The semigroup $\left(T_2(t)\right)_{t\geq 0}$ on $E_2$ associated with $\ea$ enjoys the following properties:
\begin{itemize}
\item $\left(T_2(t)\right)_{t\geq 0}$ is positive if and only if the matrix $M$ has positive off-diagonal -- that is, if it generates a positive matrix semigroup $(\e^{tM})_{t\geq 0}$;
\item Since $M$ is negative semidefinite, the semigroup $\left(T_2(t)\right)_{t\geq 0}$ is contractive on $E_{\infty}$ if and only if
\[b_{ii}+\sum_{k\neq i}|b_{ik}|\leq 0,\quad i=1,\dots ,n,\]
that is $(\e^{tM})_{t\geq 0}$ is $\ell^{\infty}$-contractive.
\end{itemize}
\end{proposition}
\begin{proof}
It follows using analogous techniques as in the proof of \cite[Theorem 3.5]{Mu07} and \cite[Lemma 4.1, Proposition 5.3]{MR07}
\end{proof}

To obtain the desired extension of the semigroup on $L^p$-spaces, we assume the following on the matrix $M.$

\begin{asum}\label{as:M}
For the matrix $M=\left(b_{ij}\right)_{n\times n}$ we assume the following properties:
\begin{enumerate}
	\item $M$ satisfies Assumption \ref{asum:M1};
	\item For $i \neq k,$ $b_{ik}\geq 0$, that is, $M$ has positive off-diagonal;
	\item \[\sum_{k\neq i}b_{ik}\leq -b_{ii},\quad i=1,\dots ,n.\]
	that is, the matrix is \emph{diagonally dominant}.
\end{enumerate}
\end{asum}

\begin{proposition}\label{prop:sgrextend}
If $M$ satisfies Assumptions \ref{as:M} then the semigroup $(T_2(t))_{t\geq 0}$ extends to a family of compact, contractive, positive one-parameter semigroups $(T_p(t))_{t\geq 0}$ on $E_p$, $1\leq p\leq \infty$. Such semigroups are strongly continuous if $p\in[1,\infty)$, and analytic of angle $\frac{\pi}{2}-\arctan\frac{\vert p-2\vert}{2\sqrt{p-1}}$ for $p\in(1,\infty)$.

Moreover, the spectrum of $A_p$ is independent of $p$, where  $A_p$ denotes the generator of $(T_p(t))_{t\geq 0}$, $1\leq p\leq \infty$.
\end{proposition}
\begin{proof}
It follows by \cite[Section 7.2]{Ar04} as in \cite[Theorem 4.1]{Mu07} and \cite[Corollary 5.6]{MR07}.
\end{proof}

We also can prove that the generators of the semigroups in the spaces $E_p,\, 1\leq p \leq \infty$ have in fact the same form as in $E_2$, with appropriate domain.

\begin{lemma}\label{lem:opAp}
For all $p\in [1,\infty]$ the generator $A_p$ of the semigroup
$(T_p(t))_{t\geq 0}$ is given by the operator defined in~\eqref{eq:opAmax} with
domain
\begin{equation}\label{eq:domAp}
D(A_p)=\left\{u\in \prod _{j=1}^m W^{2,p}(0,1;\mu_j dx)\cap D(L):MLu=Cu\right\}.
\end{equation}
In particular, $A_p$ has compact resolvent for $p\in [1,\infty]$.
\end{lemma}
\begin{proof}
See \cite[Proposition 4.6]{Mu07} and \cite[Lemma 5.7]{MR07}.
\end{proof}

As a summary we obtain the following theorem.

\begin{theorem}\label{theo:determLp}
The first order problem~\eqref{netcp} is well-posed on $E_p$, $p\in[1,\infty)$, i.e., for all initial data
$\mathsf{u}\in E_p$ the problem~\eqref{netcp} admits a unique mild solution that continuously depends on the initial data.
\end{theorem}

\section{The stochastic Allen-Cahn equation on networks}\label{sec:stochnetwork}

\subsection{An abstract stochastic Cauchy problem}\label{subsec:KVN}

Let $(\Omega,\mathscr{F},\mathbb{P})$ is a complete probability space endowed with a right continuous filtration $\mathbb{F}=(\mathscr{F}_t)_{t\in [0,T]}$. Let $(W_H(t))_{t\in [0,T]}$ be a cylindrical Wiener process, defined on $(\Omega,\mathscr{F},\mathbb{P})$, in some Hilbert space $H$ with respect to the filtration $\mathbb{F}$; that is,
  $(W_H(t))_{t\in [0,T]}$ is $(\mathscr{F}_t)_{t\in [0,T]}$-adapted and for all $t>s$, $W_H(t)-W_H(s)$ is independent of $\mathscr{F}_s$.
To be able to handle the stochastic Allen-Cahn equation on networks, first we cite a result of M.~Kunze and J.~van Neerven, regarding the following abstract equation
\begin{equation}\tag{SCP}\label{eq:SCP}
\left\{
\begin{aligned}
d X(t)&=[A X(t)+F(t,X(t))]dt+G(t,X(t))d W_{H}(t)\\
X(0)&=\xi,
\end{aligned}
\right.
\end{equation}
see \cite[Section 3]{KvN12}.
If we assume that $(A, D(A))$ generates a strongly continuous, analytic semigroup $S$ on the Banach space $E$ with $\Vert S(t)\Vert\leq K\e^{\omega t}$, $t\geq 0$ for some $K\geq 1$ and $\omega\in\real$, then for $\omega'>\omega$ the fractional powers $(\omega'-A)^{\alpha}$ are well-defined for all $\alpha\in(0,1).$ In particular, the fractional domain spaces
\begin{equation}\label{eq:fractdom}
E^{\alpha}\coloneqq D((\omega'-A)^{\alpha}),\quad \|v\|_{\alpha}\coloneqq\|(\omega'-A)^{\alpha}v\|,\quad v\in D((\omega'-A)^{\alpha})
\end{equation}
are Banach spaces. It is well-known (see e.g. \cite[$\mathsection$II.4--5.]{EN00}), that up to equivalent norms, these spaces are independent of the choice of $\omega'.$

For $\alpha\in(0,1)$ we define the extrapolation spaces $E^{-\alpha}$ as the completion of $E$ under the norms $\|v\|_{-\alpha}\coloneqq\|(\omega'-A)^{-\alpha}v\|$, $v\in E$. These spaces are independent of $\omega'>\omega$ up to an equivalent norm.

We fix $E^0\coloneqq E$.

\begin{remark}\label{rem:omega0}
If $\omega=0$ (hence, the semigroup $S$ is bounded), then by \cite[Proposition 3.1.7]{Haase06}
we can choose $\omega'=0$. That is,
\[E^{\alpha}\cong D((-A)^{\alpha}),\quad \alpha\in [0,1),\]
when $D((-A)^\alpha)$ is equipped with the graph norm.
\end{remark}

To obtain the desired result for the solution of \eqref{eq:SCP}, one has to impose the following assumptions for the mappings in (SCP). These are -- in the first and third cases slightly simplified versions of -- Assumptions (A1), (A5), (A4), (F'), (F'') and (G'') in \cite{KvN12}. Let $B$ be a Banach space, $\|\cdot \|$ will denote $\|\cdot\|_{B}$. For $u\in B$ we define the \emph{subdifferential of the norm at} $u$ as the set
\begin{equation}\label{eq:subdiff}
\partial\|u\|\coloneqq\left\{u^*\in B^*:\|u^*\|=1\text{ and }\langle u,u^*\rangle=1\right\}
\end{equation}
which is not empty by the Hahn-Banach theorem. Furthermore, let $E$ be a UMD Banach space of type $2$.

\begin{assum}$ $\label{assum:mainvN}
\begin{enumerate}
	\item $(A, D(A))$ is densely defined, closed and sectorial on $E$.
	\item For some $0\leq \theta <\frac{1}{2}$ we have continuous, dense embeddings \[E^{\theta}\hookrightarrow B\hookrightarrow E.\]
	\item Let $S$ be the strongly continuous analytic semigroup generated by $(A, D(A))$. Then $S$ restricts to a strongly continuous contraction semigroup $S^{B}$ on $B$, in particular, $A|_{B}$ is dissipative.
	\item The map $F\colon [0,T]\times\Omega\times B\to B$ is locally Lipschitz continuous in the sense that for all $r>0$, there exists a constant $L_{F}^{(r)}$ such that
	      \[\left\|F(t,\omega,u)-F(t,\omega,v)\right\|\leq L_{F}^{(r)}\|u-v\|\]
	      for all $\|u\|,\|v\|\leq r$ and $(t,\omega)\in [0,T]\times \Omega$ and there exists a constant $C_{F,0}\geq 0$ such that
				\[\left\|F(t,\omega,0)\right\|\leq C_{F,0},\qquad t\in[0,T],\; \omega\in\Omega.\]
				Moreover, for all $u\in B$ the map $(t,\omega)\mapsto F(t,\omega,u)$ is strongly measurable and adapted.\\
				Finally, for suitable constants $a,b\geq 0$ and $N\geq 1$ we have
				\[\langle A u+F (t,u+v), u^*\rangle\leq a(1+\|v\|)^N+b\|u\|\]
				for all $u\in D(A|_{B})$, $v\in B$ and $u^*\in\partial\|u\|,$ see \eqref{eq:subdiff}.
	\item There exist constants $a'',\, b'',\, m'>0$ such that the function $F\colon [0,T]\times\Omega\times B\to B$ satisfies
	      \[\langle F (t,\omega, u+v)-F (t,\omega, v), u^*\rangle\leq a''(1+\|v\|)^{m'}-b''\|u\|^{m'} \]
				for all $t\in[0,T]$, $\omega\in\Omega$, $u,v\in B$ and $u^*\in\partial\|u\|,$ and
				\[\left\|F(t,v)\right\|\leq a''(1+\|v\|)^{m'}\]
				for all $v\in B.$
	\item Let $\gamma(H,E^{-\kappa_G})$ denote the space of $\gamma$-radonifying operators from $H$ to $E^{-\kappa_G}$ for some $0\leq \kappa_G<\frac{1}{2}$, see e.g. \cite[Section 3.1]{KvN12}.
	      Then the map $G\colon [0,T]\times\Omega\times B\to \gamma(H,E^{-\kappa_G})$ is locally Lipschitz continuous in the sense that for all $r>0$, there exists a constant $L_{G}^{(r)}$ such that
	      \[\left\|G(t,\omega,u)-G(t,\omega,v)\right\|_{\gamma(H,E^{-\kappa_G})}\leq L_{G}^{(r)}\|u-v\|\]
	      for all $\|u\|,\|v\|\leq r$ and $(t,\omega)\in [0,T]\times \Omega$.
				Moreover, for all $u\in B$ and $h\in H$ the map $(t,\omega)\mapsto G(t,\omega,u)h$ is strongly measurable and adapted.\\[0.2cm]
				Finally, $G$ is of linear growth, that is, for suitable constant $c',$
				\[\left\|G(t,\omega,u)\right\|_{\gamma(H,E^{-\kappa_G})}\leq c'\left(1+\|u\|\right)\]
				for all $(t,\omega,u)\in [0,T]\times \Omega\times B.$
\end{enumerate}
\end{assum}

Recall that a \emph{mild solution} of \eqref{eq:SCP} is a solution of the following implicit equation
\begin{align}\label{eq:mildsol}
X(t)&=S(t)\xi+\int_0^tS(t-s)F(s,X(s))\ds+\int_0^tS(t-s)G(s,X(s))\,dW_H(s)\notag\\
&\eqqcolon S(t)\xi+S\ast F(\cdot,X(\cdot))(t)+S\diamond G(\cdot,X(\cdot))(t)
\end{align}
where
\[S\ast f(t)=\int_0^tS(t-s)f(s)\ds\]
denotes the "usual" convolution, and
\[S\diamond g(t)=\int_0^tS(t-s)g(s)\,dW_{H}(s)\]
denotes the stochastic convolution with respect to $W_{H}.$

The result of Kunze and van Neerven that will be useful for our setting is the following. We note that this was first proved in \cite[Theorem 4.9]{KvN12} but with a typo in the statement which was later corrected in the recent arXiv preprint \cite[Theorem 4.9]{KvN19}.

\begin{theorem}\cite[Theorem 4.9]{KvN19}\label{theo:KvN4.9}
Suppose that Assumptions \ref{assum:mainvN} hold and let $2<q<\infty$, $0\leq \theta<\frac{1}{2}$, $0\leq \kappa_G<\frac{1}{2}$ satisfy 
\[\theta+\kappa_G<\frac{1}{2}-\frac{1}{q}.\]
Then for all $\xi\in L^q(\Omega,\mathscr{F}_0,\mathbb{P}; B)$ there exists a unique global mild solution 
\[X\in L^q\left(\Omega,C([0,T];B)\right)\] 
of \eqref{eq:SCP}. Moreover, for some constant $C>0$ we have
\[\mathbb{E}\|X\|^q_{C([0,T];B)}\leq C\cdot\left(1+\mathbb{E}\|\xi\|^q\right).\]
\end{theorem}

\subsection{Preparatory results}\label{subsec:prep}

In order to apply the abstract result of Theorem \ref{theo:KvN4.9} to the stochastic Allen-Cahn equation on a network we need to prove some preparatory results using the setting of Section \ref{sec:determnetwork}.

On the edges of the graph $G$ we will consider continuous functions that satisfy the continuity condition in the vertices, see Subsection \ref{subsec:systeq}. We will refer to such functions as \emph{continuous functions on the graph} $G$ and denote them by $C(G).$

\begin{definition}\label{def:b}
We define
\begin{equation}\label{eq:CG}
C(G)\coloneqq D(L),
\end{equation} 
see \eqref{eq:Ldef}, which can be looked at as the Banach space of all continuous functions on the graph $G$, hence the norm on $C(G)$ can be defined as
\begin{equation}\label{eq:normCG}
\|u\|_{C(G)}=\max_{j=1,\dots ,m}\sup_{[0,1]}|u_j|,\quad u\in C(G).
\end{equation}
This space will play the role of the space $B$ in our setting, hence we set
\begin{equation}\label{eq:spB}
B\coloneqq C(G)\text{ and }\|\cdot\|_{C(G)}\coloneqq\|\cdot\|_{B}.
\end{equation}
\end{definition} 

We will show that for $\theta$ big enough the continuous, dense embeddings 
\[E_p^{\theta}\hookrightarrow B\hookrightarrow E_p\]
hold, where 
\begin{equation}\label{eq:fractdomEp}
E_p^{\theta}\text{ is defined for the operator }A_p\text{ on the Banach space }E_p\text{ as in \eqref{eq:fractdom}.} 
\end{equation}To do so, we first need a technical lemma, and define the maximal operator on $E_p$ as
\begin{equation}\label{eq:opApmax}
A_{p,max}\coloneqq\begin{pmatrix}
\frac{d}{dx}\left(c_1 \frac{d}{dx}\right)-p_m & & 0\\
 & \ddots &\\
0 & & \frac{d}{dx}\left(c_m \frac{d}{dx}\right)-p_m \\
\end{pmatrix}
\end{equation}
with domain
\begin{equation}\label{eq:domApmax}
D(A_{p,max})\coloneqq \left(\prod _{j=1}^m W^{2,p}(0,1;\mu_j dx)\right)\cap D(L),
\end{equation}
see \eqref{eq:opAmax} \eqref{eq:domAmax} in $E_2.$ Hence, the domain of $A_{p,max}$ only contains the continuity condition in the nodes.

Furthermore, define
\begin{equation}\label{eq:W0G}
W_0(G)\coloneqq \prod _{j=1}^m W_{0}^{2,p}(0,1;\mu_j dx),
\end{equation}
where 
\[W_{0}^{2,p}(0,1;\mu_j dx)=W^{2,p}(0,1;\mu_j dx)\cap W_{0}^{1,p}(0,1;\mu_j dx),\qquad j=1,\dots ,m.\] 
That is, $W_0(G)$ contains such vectors of functions that are twice weakly differentiable on each edge and continuous on the graph with Dirichlet boundary conditions.

\begin{lemma}\label{lem:ApmaxW0G}
\begin{equation}
D(A_{p,max})\cong W_0(G)\times\real^n,
\end{equation}
where the isomorphism is taken for $D(A_{p,max})$ equipped with the operator graph norm.
\end{lemma}
\begin{proof}
We will use the setting of \cite{Gr87} for $A=A_{p,max}$, $X=E_p$ and the boundary operator $L:D(L) \subset E_p\to\real^n \eqqcolon Y$. Denote
\[A_0\coloneqq {A_{p,max}\mid}_{\ker L},\]
which is the operator \eqref{eq:opApmax} with Dirichlet boundary conditions. Hence, it is a generator on $E_p.$ Clearly
\begin{equation}\label{eq:A0W0}
D(A_0)=W_0(G)
\end{equation}
holds. 

We now choose $\la\in\rho(A_0)$. Using \cite[Lemma 1.2]{Gr87} we have that
\[D(A_{p,max})=D(A_0)\oplus \ker(\la-A_{p,max}).\]
Furthermore, the map
\begin{equation}\label{eq:LkerAiso}
L\colon \ker(\la-A_{p,max})\to\real^n
\end{equation}
is an onto isomorphism, having the inverse
\[D_{\la}\coloneqq (L\mid_{\ker(\la-A_{p,max})})^{-1}\colon \real^n\to \ker(\la-A_{p,max})\]
called \emph{Dirichlet-operator}, see \cite[(1.14)]{Gr87}. By \cite[(1.15)]{Gr87},
\[D_{\la} L\colon D(A_{p,max})\to \ker(\la-A_{p,max})\]
is the projection in $D(A_{p,max})$ onto $\ker(\la-A_{p,max})$ along $D(A_0)$. Since $D_{\la}L$ is continuous, by the properties of the direct sum, see e.g. \cite[Theorem 2.5]{RD18}, we obtain that
\[D(A_{p,max})\cong D(A_0)\times \ker(\la-A_{p,max})\]
holds. Now using \eqref{eq:A0W0} and that \eqref{eq:LkerAiso} is an isomorphism, the claim follows.
\end{proof}

\begin{lemma}\label{lem:Biso}
For the space $B$ defined in \eqref{eq:spB}
\[B\cong \left(C_0[0,1]\right)^m\times\real^n\]
holds.
\end{lemma}
\begin{proof}
Let $u\in B$ arbitrary and $r\coloneqq Lu\in\real^n.$ We can define the unique $v^u\in B$ such that $v^u_j$ is a first order polynomial for each $j=1,\dots ,m$ taking values
\[v^u_j(\mv_i)=r_i,\quad \text{ for }\me_j\in\Gamma(\mv_i)\; j=1,\dots ,m,\; i=1,\dots ,n.\]
Then $Lv^u=r$ and 
\[u-v^u\in \left(C_0[0,1]\right)^m.\]
Denote
\[B_1\coloneqq\left\{v^u:u\in B\right\}\subset B\]
a closed subspace. Clearly,
\[\left(C_0[0,1]\right)^m\cap B_1=\{0_{B}\}\]
and if $u\in B$ then $u=(u-v^u)+v^u$ with $u-v^u\in \left(C_0[0,1]\right)^m$ and $v^u\in B_1$. Hence
\[B= \left(C_0[0,1]\right)^m\oplus B_1.\]
By the construction of $v^u$ follows that since $L:B\to\real^n$ is onto,
\[L\mid_{B_1}\colon B_1\to\real^n\] 
is a bijection. The operator $L\mid_{B_1}$ is also bounded for the norm of $B$ induced on $B_1$. Hence, by the open mapping theorem, it is an isomorphism. Denoting its inverse by
\[L_1\coloneqq \left(L\mid_{B_1}\right)^{-1}\colon \real^n\to B_1,\]
we obtain that
\[L_1L\colon B\to B_1\]
is the continuous projection from $B$ onto $B_1$ along $\left(C_0[0,1]\right)^m.$ Hence, we can use \cite[Theorem 2.5]{RD18} and obtain
\[B\cong \left(C_0[0,1]\right)^m\times\real^n.\]
\end{proof}

\begin{corollary}\label{cor:fractionalspaceincl}
Let $E_p^{\theta}$ defined in \eqref{eq:fractdomEp}. If $\theta> \frac{1}{2p}$ then the following continuous, dense embeddings are satisfied:
\begin{equation}\label{eq:Ethetaincl}
E_p^{\theta}\hookrightarrow B\hookrightarrow E_p.
\end{equation} 
\end{corollary}
\begin{proof}
By Proposition \ref{prop:sgrextend} the operator $(A_p,D(A_p))$ generates a positive, contraction semigroup on $E_p$.
It follows from \cite[Theorem in $\mathsection$4.7.3]{Ar04} and \cite[Proposition in $\mathsection$4.4.10]{Ar04} that for the complex interpolation spaces
\begin{equation}\label{eq:inclpf2}
D((\omega'-A_p)^{\theta})\cong [D(\omega'-A_p),E_p]_{\theta}
\end{equation}
holds for any $\omega'>0$. Therefore,
\begin{equation}
E_p^{\theta}=D((\omega'-A_p)^{\theta})\cong [D(\omega'-A_p),E_p]_{\theta}\cong [D(A_p),E_p]_{\theta}.
\end{equation}
Defining $(A_{p,max},D(A_{p,max}))$ as in \eqref{eq:opApmax}, \eqref{eq:domApmax}
we have that
\[D(A_p)\hookrightarrow D(A_{p,max})\]
holds. Hence
\begin{equation}\label{eq:inclpf3}
E_p^{\theta}\hookrightarrow \left[D(-A_{p,max}),E_p\right]_{\theta}.
\end{equation}
By Lemma \ref{lem:ApmaxW0G},
\begin{equation}\label{eq:inclpf4}
D(-A_{p,max})\cong W_0(G)\times\real^n
\end{equation}
holds, where $W_0(G)$ is defined in \eqref{eq:W0G}. Since $E_p\cong E_p\times\{0_{\real^n}\}$, using general interpolation theory, see e.g. \cite[Section 4.3.3]{Triebel78}, we have that for $\theta>\frac{1}{2p}$
\[\left[W_0(G)\times\real^n,E_p\times\{0_{\real^n}\}\right]_{\theta}\hookrightarrow \left(\prod _{j=1}^m W_{0}^{2\theta,p}(0,1;\mu_j dx)\right)\times \real^n.\]
Thus, by \eqref{eq:inclpf3} and \eqref{eq:inclpf4}
\begin{equation}\label{eq:interpol}
E_p^{\theta} \hookrightarrow \left(\prod _{j=1}^m W_{0}^{2\theta,p}(0,1;\mu_j dx)\right)\times\real^n
\end{equation}
holds. Hence,
\begin{equation}\label{eq:inclpf5}
E_p^{\theta} \hookrightarrow \left(C_0[0,1]\right)^m\times\real^n
\end{equation}
is true.
Applying Lemma \ref{lem:Biso} we obtain that for $\theta>\frac{1}{2p}$
\begin{equation}\label{eq:inclpf6}
E_p^{\theta}\hookrightarrow B
\end{equation}
is satisfied. The continuity of the embedding $B\hookrightarrow E_p$ is clear. It follows from Proposition \ref{prop:mcAonC} that $D(A_p)$ is a dense subspace of $B$ and then so is $E_p^{\theta}$ for $\theta>\frac{1}{2p}$. Since $B\cong \left(C_0[0,1]\right)^m\times\real^n$ by Lemma \ref{lem:Biso} and $E_p\cong E_p\times\{0_{\real^n}\}$, the space $B$ is also dense in $E_p$ and the claim follows.
\end{proof}

In the following we will prove that the part of the operator $(A_p,D(A_p))$ in $B$ is the generator of a strongly continuous semigroup on $B.$

\begin{proposition}\label{prop:mcAonC}
The part of $(A_p,D(A_p))$ in $B$ generates a positive strongly continuous semigroup of contractions on $B$.
\end{proposition}
\begin{proof}
1. We first prove that the semigroup $(T_p(t))_{t\geq 0}$ leaves $B$ invariant. We take $u\in B\subset E_p$ and use that $(T_p(t))_{t\geq 0}$ is analytic on $E_p$ (see Proposition \ref{prop:sgrextend}). . Hence, $T_p(t)u\in D(A_p).$ The explicit form \eqref{eq:domAp} of $D(A_p)$ shows that $D(A_p)\subset B$ and hence also
\[T_p(t)u\in B\] holds.

2. In the next step we prove that $(T_p(t)|_{B})_{t\geq 0}$ is a strongly continuous semigroup. By \cite[Proposition I.5.3]{EN00}, it is enough to prove that there exist $K>0$ and $\delta>0$ and a dense subspace $D\subset B$ such that 
\begin{enumerate}[(a)]
	\item $\|T_p(t)\|_{B}\leq K$ for all $t\in[0,\delta]$, and
	\item $\lim_{t\downarrow 0}T_p (t)u=u$ for all $u\in D$.
\end{enumerate}
To verify (a), we obtain by Proposition \ref{prop:sgrextend} that for $u\in B$
\[\|T_p(t)u\|_{B}= \|T_p(t)u\|_{E_{\infty}}= \|T_{\infty}(t)u\|_{E_{\infty}}\leq \|u\|_{E_{\infty}}= \|u\|_{B},\]
hence
\[\|T_p(t)\|_{B}\leq 1=:K,\quad t\geq 0.\]
To prove (b) we first set $p=2$. Taking $\omega>0$ arbitrary, we obtain that the form 
\[\ea_{\omega}(u,v)\coloneqq \ea(u,v)+\omega \cdot\langle u,v\rangle_{E_2},\quad u,v\in V\]
is coercive, symmetric and continuous, see \cite[Remark 7.3.3]{Haase06} and Proposition \ref{prop:eaprop}. For the form-domain $V$ defined in \eqref{eq:domform}, equipped with the usual $(H^1(0,1))^m$-norm, we have that
\[V=D((\omega-A_2)^{\sfrac{1}{2}})\]
holds with equivalence of norms (see e.g.~\cite[Proposition 5.5.1]{Ar04}). We also have
\begin{equation}\label{eq:isom}
V=D((\omega-A_2)^{\sfrac{1}{2}})=D((-A_2)^{\sfrac{1}{2}})
\end{equation}
with equivalent norms, where we used  \cite[Proposition 3.1.7]{Haase06} for the second equality and norm equivalence. Notice that the subspace $(C^{\infty}[0,1])^m\cap B$ (the infinitely many times differentiable functions on the edges that are continuous across the vertices) is contained in $V$ and is dense in $B$ by the Stone--Weierstrass theorem. Hence, $V$ and thus $D((-A_2)^{\sfrac{1}{2}})$ is dense in $B$. Defining $D\coloneqq D((-A_2)^{\sfrac{1}{2}})$, for $u\in D$ there exist $C_1, C_2>0$ such that
\begin{align}
\|T_2 (t)u-u\|_{B}&\leq C_1\cdot \|T_2 (t)u-u\|_{\left(H^1(0,1)\right)^m} \\
&\leq C_2\cdot \left(\|T_2 (t)(-A_2)^{\sfrac{1}{2}}u-(-A_2)^{\sfrac{1}{2}}u\|_{E_{2}}+\|T_2 (t) u-u\|_{E_{2}}\right)\to 0,\quad t\downarrow 0.
\end{align}
In the first inequality we have used Sobolev embedding and in the second one the norm equivalence in \eqref{eq:isom} and the the fact that $T_2 (t)$ and $(-A_2)^{\sfrac{1}{2}}$ commute on $D((-A_2)^{\sfrac{1}{2}})$. Summarizing 1. and 2., and using that clearly $B$ is continuously embedded in $E_p$, we can apply \cite[Proposition in Section II.2.3]{EN00} for $(A_2,D(A_2))$ and $Y=B$, and obtain that the part of $(A_2,D(A_2))$ in $B$ generates a positive strongly continuous semigroup of contractions on $B$. Since the semigroups in Proposition \ref{prop:sgrextend} are consistent, the same is true for $(T_p(t))_{t\geq 0}$ for any $p\in [1,\infty].$
\end{proof}

\begin{corollary}\label{prop:determcont}
The first order problem \eqref{netcp} is well-posed on $B$, i.e., for all initial data $\mathsf{u}\in B$ the problem~\eqref{netcp} admits a unique mild solution that continuously depends on the initial data.
\end{corollary}

\subsection{Main results}\label{subsec:main}

In this subsection we first apply the above results to the following stochastic evolution equation, based on \eqref{netcp}. This corresponds to a slightly more general version of \eqref{eq:stochnet}, see \eqref{eq:stochneteq} later.

Let $(\Omega,\mathscr{F},\mathbb{P})$ be a complete probability space endowed with a right-continuous filtration $\mathbb{F}=(\mathscr{F}_t)_{t\in [0,T]}$ for some $T>0$ given. We consider the problem
\begin{equation}\label{eq:stochsys}
\left\{\begin{aligned}
\dot{u}_j(t,x)& = (c_j u_j')'(t,x)-p_j(x)u_j(t,x)&&\\
&\quad + f_j(t,x,u_j(t,x))&&\\
&\quad +g_j(t,x,u_j(t,x))\dfrac{\partial w_j}{\partial t}(t,x), &&t\in(0,T],\; x\in(0,1),\; j=1,\dots,m, && (a)\\
u_j(t,\mv _i)& =u_\ell (t,\mv _i)\eqqcolon q_i(t), &&t\in(0,T],\; \forall j,\ell\in \Gamma(\mv _i),\; i=1,\ldots,n,&& (b)\\
\left[M q(t)\right]_{i} & = -\sum\nolimits_{j=1}^m \phi_{ij}\mu_{j} c_j(\mv_i) u'_j(t,\mv_i), &&t\in(0,T],\; i=1,\ldots,n,&& (c)\\
u_j(0,x)& =\mathsf{u}_{j}(x), &&x\in [0,1],\; j=1,\dots,m, && (d)
\end{aligned}
\right.
\end{equation}
where  $\frac{\partial w_j}{\partial t}$, $j=1,\dots ,m$, are independent space-time white noises on $[0,1]$; written as formal derivatives of independent cylindrical Wiener-processes $(w_j(t))_{t\in [0,T]}$, defined on $(\Omega,\mathscr{F},\mathbb{P})$, in the Hilbert space $L^2(0,1; \mu_j dx)$  with respect to the filtration $\mathbb{F}$.  

 The functions $f_j\colon[0,T]\times \Omega\times [0,1]\times \real\to\real$ are polynomials of the form
\begin{equation}\label{eq:fjdef}
f_j(t,\omega,x,\eta)=-a_{j,2k+1}(t,\omega,x)\eta^{2k+1}+\sum_{l=0}^{2k}a_{j,l}(t,\omega,x)\eta^l,\quad \eta\in\real,\, j=1,\dots ,m
\end{equation} 
for some fixed integer $k$. For the coefficients we assume that there are constants $0<c\leq C<\infty$ such that
\begin{equation}\label{eq:assa_j}
c\leq a_{j,2k+1}(t,\omega,x)\leq C,\;\left|a_{j,l}(t,\omega,x)\right|\leq C,\text{ for all }j=1,\dots ,m,\;l=0,2,\dots ,2k,
\end{equation}
for all $x\in [0,1]$,  $t\in[0,T]$ and almost all $\omega\in\Omega$, see \cite[Example 4.2]{KvN12}. The coefficients $a_{j,l}\colon [0,T]\times \Omega\times [0,1]\to\real$ are jointly measurable and adapted in the sense that for each $j$ and $l$ and for each $t\in[0,T]$, the function $a_{j,l}(t,\cdot)$ is $\mathscr{F}_t\otimes \mcB_{[0,1]}$-measurable, where $\mcB_{[0,1]}$ denotes the sigma-algebra of the Borel sets on $[0,1].$

We further assume a technical assumption regarding the graph structure that will play and important role in our setting.
\begin{asum}\label{asum:F}
For the coefficients in \eqref{eq:fjdef} we assume that
\[\left(a_{1,l}(t,\omega,\cdot),\dots ,a_{m,l}(t,\omega,\cdot)\right)^{\top}\in B\text{ for all }l=1,\dots ,2k+1,\]
$t\in[0,T]$ and almost all $\omega\in\Omega$.
\end{asum}

\begin{remark}\label{rem:exfj}
If the coefficients in \eqref{eq:fjdef} do not depend on $j$ -- that is, they are the same on different edges --, and satisfy
\[a_{l}(t,\omega,\cdot)=a_{j,l}(t,\omega,\cdot)\in C[0,1],\quad t\in[0,T],\omega\in\Omega,\quad j=1,\dots m,\; l=1,\dots ,2k+1\]
and
\[a_{l}(t,\omega,0)=a_l(t,\omega,1),\text{ for all }l=1,\dots 2k+1,\]
then Assumption \ref{asum:F} is fulfilled. This is the case e.g. if $a_l's$ are constant (not depending on $x$).
\end{remark}

For the functions $g_j$ we assume
\begin{align}\label{eq:gidef}
&g_j\colon [0,T]\times \Omega\times [0,1]\times\real\to \real,\quad j=1,\dots ,m \text{ are locally Lipschitz continuous}\notag\\
&\text{and of linear growth in the fourth variable,}\notag\\
&\text{uniformly with respect to the first three variables.}
\end{align}
We further assume that the functions are jointly measurable and adapted in the sense that for each $j$ and $t\in[0,T]$, $g_j(t,\cdot)$ is $\mathscr{F}_t\otimes \mcB_{[0,1]}\otimes \mcB_{\real}$-measurable, where $\mcB_{[0,1]}$ and $\mcB_{\real}$ denote the sigma-algebras of the Borel sets on $[0,1]$ and $\real$, respectively.

The above assumptions on the coefficients on the edges, except for Assumption \ref{asum:F} which is specific for the graph setting, are analogous to those in  \cite[Section 5]{KvN12} and \cite[Section 5]{KvN19}.
\medskip

To handle system \eqref{eq:stochsys}, we rewrite it in the form of the abstract stochastic Cauchy-problem \eqref{eq:SCP}. To do so, we specify the functions appearing in \eqref{eq:SCP} corresponding to \eqref{eq:stochsys}.

The operator $(A, D(A))=(A_p, D(A_p))$ will be the generator of the strongly continuous analytic semigroup $S\coloneqq (T_p(t))_{t\geq 0}$ on the Banach space $E\coloneqq E_p$ for some large $p\geq 2$, see Proposition \ref{prop:sgrextend} and Lemma \ref{lem:opAp}. Hence, $E$ is a UMD space of type $2$. 

For the function $F\colon [0,T]\times \Omega\times B\to B$ we have
\begin{equation}\label{eq:Fdef}
F(t,\omega,u)(s)\coloneqq  \left(f_1(t,\omega,s,u_1(s)),\dots ,f_m(t,\omega,s,u_m(s))\right)^{\top}, \quad s\in[0,1].
\end{equation}
Since $B$ is an algebra, Assumption \ref{asum:F} assures that $F$ maps $[0,T]\times \Omega\times B$ into $B.$

To define the operator $G$ we argue in analogy with \cite[Section 5]{KvN19}. First define 
\[
H\coloneqq E_2
\]
the product $L^2$-space, see \eqref{eq:E2}, which is a Hilbert space. We further define the multiplication operator $\Gamma\colon [0,T]\times B\to\mathcal{L}(H)$ as
\begin{equation}\label{eq:Gammadef}
\left[\Gamma(t,u)h\right](s)\coloneqq\begin{pmatrix}
	 g_1(t,s,u_1(s)) & \hdots & 0\\
	\vdots & \ddots & \vdots\\
	 0 & \hdots  &g_m(t,s,u_m(s))
	\end{pmatrix}\cdot\begin{pmatrix}
	 h_1(s)\\
	\vdots\\
	h_m(s)
	\end{pmatrix},\quad s\in(0,1),
\end{equation}
for $u\in B$, $h\in H$. Because of the assumptions \eqref{eq:gidef} on the functions $g_j$, $\Gamma$ clearly maps into $\mathcal{L}(H).$ 

Let $(A_2,D(A_2))$ be the generator on $H=E_2,$ see Proposition \ref{prop:determL2}, and pick $\kappa_G\in(\frac{1}{4},\frac{1}{2})$. 
By \eqref{eq:interpol} in the proof of Corollary \ref{cor:fractionalspaceincl} we have that there exists a continuous embedding
\begin{equation}\label{eq:imath}
\imath\colon E_2^{\kappa_G} \to \left(\prod _{j=1}^m H_{0}^{2\kappa_G}(0,1;\mu_j dx)\right)\times\real^n\eqqcolon \mathcal{H},
\end{equation}
where $\mathcal{H}$ is a Hilbert space. Applying the steps \eqref{eq:inclpf5} and \eqref{eq:inclpf6} of Corollary \ref{cor:fractionalspaceincl} we obtain that $\mathcal{H}\hookrightarrow B$ holds, and by \eqref{eq:Ethetaincl}, there exists a continuous embedding
\begin{equation}\label{eq:jmath}
\jmath\colon \mathcal{H}\to E_p
\end{equation}
for $p\geq 2$ arbitrary.

Let $\nu>0$ arbitrary and define now $G$ by
\begin{equation}\label{eq:Gdef}
(\nu-A_p)^{-\kappa_G}G(t,u)h\coloneqq \jmath\, \imath\,   (\nu-A_2)^{-\kappa_G}\Gamma(t,u)h,\quad u\in B,\; h\in H.
\end{equation}
\begin{proposition}\label{prop:Gprop}
Let $p\geq 2$ and $\kappa_G\in(\frac{1}{4},\frac{1}{2})$ be arbitrary. Then the operator $G$ defined in \eqref{eq:Gdef} maps $[0,T]\times B$ into $\gamma (H,E_p^{-\kappa_{G}})$.
\end{proposition}
\begin{proof}
We can argue as in \cite[Section 10.2]{vNVW08}. Using \cite[Lemma 2.1(4)]{vNVW08}, we obtain in a similar way as in \cite[Corollary 2.2]{vNVW08}) that $\jmath \in\gamma (\mathcal{H},E_p)$, since $2\kappa_G>\frac{1}{2}$ holds. Hence, by the definition of $G$ and the ideal property of $\gamma$-radonifying operators, the mapping $G$ takes values in $\gamma (H,E_p^{-\kappa_{G}})$.
\end{proof}

The driving noise process $W_H$ is defined by
\begin{equation}\label{eq:Wdef}
W_H(t)=\begin{pmatrix}
	w_1(t)\\
	\vdots\\
	w_m(t)
\end{pmatrix}, \,t\in [0,T],
\end{equation}
and thus $(W_H(t))_{t\in [0,T]}$ is a cylindrical Wiener process, defined on $(\Omega,\mathscr{F},\mathbb{P})$, in the Hilbert space $H$ with respect to the filtration $\mathbb{F}$.
\medskip

We will state now the result regarding system \eqref{eq:SCP} corresponding to \eqref{eq:stochsys}.

\begin{theorem}\label{theo:SCPnsolcont}
Let $F$, $G$ and $W$ defined in \eqref{eq:Fdef}, \eqref{eq:Gdef} and \eqref{eq:Wdef}, respectively. Let $q > 4$ be arbitrary. Then for every $\xi\in L^q(\Omega,\mathscr{F}_0,\mathbb{P};B)$ a unique mild solution $X$ of equation \eqref{eq:SCP} exists globally and belongs to $L^q(\Omega; C([0,T];B))$.
\end{theorem}

\begin{proof}
The condition $q > 4$ allows us to choose $2 \leq p < \infty$, $\theta\in [0,\frac{1}{2})$ and $\kappa_G\in (\frac14 ,\frac12)$ such that
\begin{equation}\label{eq:thetap}
\theta>\frac{1}{2p}
\end{equation} 
and 
\begin{equation}\label{eq:thetakappaG}
0 \leq \theta+ \kappa_G < \frac12 - \frac1q.
\end{equation} 

We will apply Theorem \ref{theo:KvN4.9} with $\theta$ and $\kappa_G$ having the properties above. To this end we have to check Assumptions \ref{assum:mainvN} for the mappings in \eqref{eq:SCP}, taking $A=A_p$ and $E=E_p$ for the $p$ chosen above. Assumption $(1)$ is satisfied because of the generator property of $A_p$, see Proposition \ref{prop:sgrextend}.
Assumption $(2)$ is satisfied since \eqref{eq:thetap} holds and we can use Corollary \ref{cor:fractionalspaceincl}.
Assumption $(3)$ is satisfied by the statement of Proposition \ref{prop:mcAonC}.
Using that the functions $f_j$ are polynomials of the 4th variable of the same degree $2k+1$ (see \eqref{eq:fjdef}), a similar computation as in \cite[Example 4.2]{KvN12} and \cite[Example 4.5]{KvN12}, using techniques from \cite[Section 4.3]{DPZ92}, shows that Assumptions $(4)$ and $(5)$ are satisfied for $F$ with $N=m'=2k+1$. By Proposition \ref{prop:Gprop}, $G$ takes values in $\gamma (H,
E_p^{-\kappa_{G}})$ with $H=E_2$ and $\kappa_G$ chosen above. Using the assumptions \eqref{eq:gidef} on the functions $g_j$ and the proof of \cite[Theorem 10.2]{vNVW08}, we obtain that $G$ is locally Lipschitz continuous and of linear growth as a map $[0,T]\times B\to \gamma (H,E_p^{-\kappa_{G}})$, hence Assumption $(6)$ holds.
\end{proof}

In the following theorem we will state a result regarding H\"older regularity of the mild solution of \eqref{eq:SCP} corresponding to \eqref{eq:stochsys}, see \eqref{eq:mildsol}.

\begin{theorem}\label{theo:Holderreg}
Let $q > 4$ be arbitrary, $\lambda, \eta>0$ and $p\geq 2$ such that $\lambda+\eta>\frac{1}{2p}$. We assume that $\xi\in L^{(2k+1)q}(\Omega;E_p^{\lambda+\eta})$, where $k$ is the constant appearing in \eqref{eq:fjdef}. If the inequality
\begin{equation}\label{eq:thmHolderreglaeta}
\lambda+\eta<\frac{1}{4}-\frac{1}{q}
\end{equation} 
is fulfilled, then the mild solution $X$ of \eqref{eq:SCP} from Theorem \ref{theo:SCPnsolcont} satisfies 
\[X\in L^q(\Omega;C^{\lambda}([0,T],E_p^{\eta})).\]
\end{theorem}

\begin{proof}
Using the continuous embedding \eqref{eq:Ethetaincl}, we have that
\[\xi\in L^{(2k+1)q}(\Omega;B)\]
holds. Since $(2k+1)q>4$, by Theorem \ref{theo:SCPnsolcont} there exists a global mild solution
\[X\in L^{(2k+1)q}(\Omega; C([0,T],B)).\]
This solution satisfies the following implicit equation (see \eqref{eq:mildsol}):
\begin{equation}\label{eq:proofmildsol}
X(t)=S(t)\xi+S\ast F(\cdot,X(\cdot))(t)+S\diamond G(\cdot,X(\cdot))(t),
\end{equation}
where $S$ denotes the semigroup generated by $A_p$ on $E_p$, $\ast$ denotes the usual convolution, $\diamond$ denotes the stochastic convolution with respect to $\mcW.$ In the following we have to estimate the $L^q(\Omega;C^{\lambda}([0,T],E_p^{\eta}))$-norm of $X$, and we will do this using the triangle-inequality in \eqref{eq:proofmildsol}. 

For the $q$th power of the first term we have
\begin{align}\label{eq:Sxi}
\mathbb{E}\|S(\cdot)\xi\|_{C^{\lambda}([0,T],E_p^{\eta})}^q &=\mathbb{E}\left(\sup_{t,s\in[0,T]}\frac{\|S(t)\xi-S(s)\xi\|_{E_p^{\eta}}}{|t-s|^{\lambda}}\right)^q\notag\\
&\leq\mathbb{E}\left(\sup_{h\in[0,T]}\frac{\|S(h)\xi-\xi\|_{E_p^{\eta}}}{|h|^{\lambda}}\right)^q\notag\\
&=\mathbb{E}\left(\sup_{h\in[0,T]}\frac{\|S(h)(-A_p)^{\eta}\xi-(-A_p)^{\eta}\xi\|_{E_p}}{|h|^{\lambda}}\right)^q.
\end{align}
By assumption, $(-A_p)^{\eta}\xi\in D((-A_p)^{\lambda})$ holds. Applying \cite[Proposition II.5.33]{EN00} we obtain that $(-A_p)^{\eta}\xi$ lies in the H\"older space of order $\lambda$ on $E_p$, denoted by $C^{\lambda}_{p}$. Hence,
\[\sup_{h\in[0,T]}\frac{\|S(h)(-A_p)^{\eta}\xi-(-A_p)^{\eta}\xi\|_{E_p}}{|h|^{\lambda}}=\|(-A_p)^{\eta}\xi\|_{\mathsf{F}_{p,\lambda}}<\infty,\]
where $\|\cdot\|_{\mathsf{F}_{p.\lambda}}$ denotes the Favard norm of order $\lambda$ on $E_p$, see \cite[Definition II.5.10]{EN00}. Furthermore, because of the continuous inclusion $D((-A_p)^{\lambda})\hookrightarrow C^{\lambda}_{p}$, we have that there exists $c=c(\lambda)$ such that
\[\|(-A_p)^{\eta}\xi\|_{\mathsf{F}_{p,\lambda}}\leq c\cdot\|(-A_p)^{\eta}\xi\|_{E_p^{\lambda}}=c\cdot \|(-A_p)^{\lambda+\eta}\xi\|_{E_p}.\]
Hence,
\[\mathbb{E}\|S(\cdot)\xi\|_{C^{\lambda}([0,T],E_p^{\eta})}^q \leq c\cdot \mathbb{E}\|(-A_p)^{\lambda+\eta}\xi\|^q_{E_p}<\infty\]
by assumption.

To estimate the $q$th power of the second term
\[\mathbb{E}\|S\ast F(\cdot,X(\cdot))\|_{C^{\lambda}([0,T],E_p^{\eta})}^q\]
we choose $\theta>\frac{1}{2p}$ such that
\[\lambda+\eta+\theta<1-\frac{1}{q}.\]
We will use \cite[Lemma 3.6]{vNVW08} with this $\theta$, $\alpha=1$, and $q$ instead of $p$, and obtain that there exist constants $C\geq 0$ and $\ve>0$ such that
\begin{equation}\label{eq:est2nd1}
\|S\ast F(\cdot,X(\cdot))\|_{C^{\lambda}([0,T],E_p^{\eta})}\leq CT^{\ve}\|F(\cdot,X(\cdot))\|_{L^q(0,T;E_p^{-\theta})}.
\end{equation}
We have to estimate the expectation of the $q$th power on the right-hand-side of \eqref{eq:est2nd1}. By Corollary \ref{cor:fractionalspaceincl} we obtain
\[B\hookrightarrow E_p\hookrightarrow E_p^{-\theta},\]
since $\theta>\frac{1}{2p}$ holds and $(\omega'-A_p)^{-\theta}$ is an isomorphism between $E_p^{-\theta}$ and $E_p$. Using this and Assumptions \ref{assum:mainvN}(5) with $m'=2k+1$ (which holds by the proof of Theorem \ref{theo:SCPnsolcont}), we have
\begin{align}
\mathbb{E}\|F(\cdot,X(\cdot))\|^q_{L^q(0,T;E_p^{-\theta})}&=\mathbb{E}\int_0^T\|F(s,X(s))\|^q_{E_p^{-\theta}}\ds\\
& \lesssim \mathbb{E}\int_0^T\|F(s,X(s))\|^q_{B}\ds\\
& \lesssim \mathbb{E}\int_0^T(1+\|X(s)\|^{(2k+1)q}_{B})\ds\\
& \lesssim 1+\mathbb{E}\sup_{t\in[0,T]}\|X(t)\|^{(2k+1)q}_{B},
\end{align}
where $\lesssim$ denotes that the expression on the left-hand-side is less or equal to a constant times the expression on the right-hand-side.
This implies that for each $T>0$ there exists $C_T>0$ such that
\begin{equation}\label{eq:SastF}
\left(\mathbb{E}\|S\ast F(\cdot,X(\cdot))\|_{C^{\lambda}([0,T],E_p^{\eta})}^q\right)^{\frac{1}{q}}
\leq C_T\cdot \left(1+\|X(t)\|_{L^{(2k+1)q}(\Omega; C([0,T],B))}^{2k+1}\right),
\end{equation}
and the right-hand-side is finite.

To estimate the stochastic convolution term in \eqref{eq:proofmildsol} we first fix $0<\alpha<\frac{1}{2}$ such that 
\[\lambda+\eta+\frac{1}{4}<\alpha-\frac{1}{q}\] 
holds. We now choose $\kappa_G\in(\frac{1}{4},\frac{1}{2})$ such that
\[\lambda+\eta+\kappa_G<\alpha-\frac{1}{q}\]
is satisfied. Applying \cite[Proposition 4.2]{vNVW08} with $\theta=\kappa_G$ and $q$ instead of $p$, we have that there exist $\ve>0$ and $C\geq 0$ such that
\begin{equation}
\mathbb{E}\left\|S\diamond G(\cdot,X(\cdot))\right\|^q_{C^{\lambda}([0,T],E_p^{\eta})} \leq C^qT^{\ve q}\int_0^T\mathbb{E}\left\|s\mapsto (t-s)^{-\alpha}G(s,X(s))\right\|^q_{\gamma(L^2(0,t;H),E_p^{-\kappa_G})}\dt.
\end{equation}
In the following we proceed similarly as done in the proof of \cite[Theorem 4.3]{KvN12}, with $N=1$ and $q$ instead of $p.$ Since $E_p^{-\kappa_G}$ is a Banach space of type $2$ (because $E_p$ is of that type), the continuous embedding
\[ L^2(0,t;\gamma(H,E_p^{-\kappa_G}))\hookrightarrow \gamma(L^2(0,t;H),E_p^{-\kappa_G})\]
holds.
Using this, Young's inequality and the properties of $G$, respectively, we obtain the following estimates
\begin{align}
\mathbb{E}\left\|S\diamond G(\cdot,X(\cdot))\right\|^q_{C^{\lambda}([0,T],E_p^{\eta})} &\lesssim T^{\ve q}\int_0^T\mathbb{E}\left\|s\mapsto (t-s)^{-\alpha}G(s,X(s))\right\|^q_{L^2(0,t;\gamma(H,E_p^{-\kappa_G}))}\dt\\
&=T^{\ve q}\mathbb{E}\int_0^T\left(\int_0^t (t-s)^{-2\alpha}\left\|G(s,X(s))\right\|^2_{\gamma(H,E_p^{-\kappa_G})}\ds\right)^{\frac{q}{2}}\dt\\
&\leq T^{\ve q}\left(\int_0^T t^{-2\alpha}\dt\right)^{\frac{q}{2}}\mathbb{E}\int_0^T \left\|G(t,X(t))\right\|^q_{\gamma(H,E_p^{-\kappa_G})}\dt \\
&\leq T^{(\frac{1}{2}-\alpha+\ve)q}(c')^q\cdot \mathbb{E}\int_0^T\left(1+\|X(t)\|_{B}\right)^q\dt\\
&\lesssim T^{(\frac{1}{2}-\alpha+\ve)q+1}(c')^q\cdot\left(1+\mathbb{E}\|X(t)\|^q_{C([0,T],B)}\right).
\end{align}
Hence, for each $T>0$ there exists constant $C'_{T}>0$ such that
\begin{equation}\label{eq:SdiamondG}
\left(\mathbb{E}\left\|S\diamond G(\cdot,X(\cdot))\right\|^q_{C^{\lambda}([0,T],E_p^{\eta})}\right)^{\frac{1}{q}}
\leq C'_{T}\cdot\left(1+\|X(t)\|_{L^{(2k+1)q}(\Omega; C([0,T],B))}\right)^{2k+1}
\end{equation}
In summary, by \eqref{eq:Sxi}, \eqref{eq:SastF} and \eqref{eq:SdiamondG}, we obtain that $X\in L^q(\Omega;C^{\lambda}([0,T],E_p^{\eta}))$ holds, hence the proof is completed.
\end{proof}

We are now in the position to finally consider \eqref{eq:stochnet}. Let
\begin{equation}\label{eq:betamax}
\beta\coloneqq \max_{1\leq j\leq m}\beta_j.
\end{equation}
We also introduce
\begin{equation}\label{eq:fdef}
f_j(\eta)\coloneqq f(\eta)=-\eta^3+\beta^2 \eta.
\end{equation}
and
\[\varrho_j\coloneqq \beta^2-\beta_j^2\geq 0,\]
With these notations, we can rewrite \eqref{eq:stochnet} in an equivalent form as
\begin{equation}\label{eq:stochneteq}
\left\{\begin{aligned}
\dot{u}_j(t,x)&= (c_j u_j')'(t,x)-\tilde{p}_j(x)u_j(t,x)&&\\
&\quad +f_j(u_j(t,x))&&\\
&\quad +g_j(t,x,u_j(t,x))\frac{\partial w_j}{\partial t}(t,x), &&t\in(0,T],\; x\in(0,1),\; j=1,\dots,m, \\
u_j(t,\mv _i)&=u_\ell (t,\mv _i)\eqqcolon q_i(t), &&t\in(0,T],\; \forall j,\ell\in \Gamma(\mv _i),\; i=1,\ldots,n,\\
\left[M q(t)\right]_{i} &= -\sum\nolimits_{j=1}^m \phi_{ij}\mu_{j} c_j(\mv_i) u'_j(t,\mv_i), && t\in(0,T],\; i=1,\ldots,n,\\
u_j(0,x)&=\mathsf{u}_{j}(x), &&x\in [0,1],\; j=1,\dots,m,  
\end{aligned}
\right.
\end{equation}
with $\tilde{p}_j(x):=p_j(x)+\varrho_j$, $j=1,\dots m.$ 

We define the operator $A_p$ on $E_p$ as in \eqref{eq:opAmax} with $\tilde{p}_j$'s instead of $p_j$'s and with domain \eqref{eq:domAp}.

\begin{theorem}\label{thm:SAC}
Let $F$, $G$ and $W$ defined in \eqref{eq:Fdef}, \eqref{eq:Gdef} and \eqref{eq:Wdef}, respectively, for the system \eqref{eq:stochneteq}. Let $q > 4$ be arbitrary. Then for every $\xi\in L^q(\Omega,\mathscr{F}_0,\mathbb{P};B)$ a unique mild solution $X$ of equation \eqref{eq:SCP} corresponding to \eqref{eq:stochneteq}, which is equivalent  to \eqref{eq:stochnet}, exists globally and belongs to $L^q(\Omega; C([0,T];B))$. Let $\lambda, \eta>0$, $p\geq 2$ be arbitrary constants such that $\lambda+\eta>\frac{1}{2p}$. If $\xi\in L^{3q}(\Omega;E_p^{\lambda+\eta})$ and the inequality
\[\lambda+\eta<\frac{1}{4}-\frac{1}{q}\] 
is fulfilled, then 
$X\in L^q(\Omega;C^{\lambda}([0,T],E_p^{\eta})).$
\end{theorem}
\begin{proof}
First note that the coefficients $\tilde{p}_j$ stay nonnegative as the constants $\varrho_j$ are nonnegative. Furthermore, the nonlinear terms $f_j=f$ in \eqref{eq:fdef} are of the form \eqref{eq:fjdef} with $k=1$ and constant coefficients. Hence,  Assumption \ref{asum:F} is fullfilled by Remark \ref{rem:exfj}. The statement then follows from Theorems \ref{theo:SCPnsolcont} and \ref{theo:Holderreg}.
\end{proof}

\subsection{Concluding remarks}

In equation (\ref{eq:stochsys}a) we could have prescribed coloured noise instead of white noise on the edges of the graph. That is, we could set
\begin{equation}\label{eq:stochsysmod}
\begin{aligned}
\dot{u}_j(t,x)&= (c_j u_j')'(t,x)-p_j(x)u_j(t,x)\\
&\quad + f_j(t,x,u_j(t,x))\\
&\quad +g_j(t,x,u_j(t,x))R_j\dfrac{\partial w_j}{\partial t}(t,x), \qquad t\in(0,T],\; x\in(0,1),\; j=1,\dots,m,
\end{aligned}
\end{equation}
with $R_j\in \gamma (L^2(0,1;\mu_j dx),L^p(0,1;\mu_j dx))$. Then we define
\[R\coloneqq \begin{pmatrix}
	 R_1 & \hdots & 0\\
	\vdots & \ddots & \vdots\\
	 0 & \hdots  &R_m
	\end{pmatrix}\in \gamma (H,E_p)\]
with $H=E_2$ and $p\geq 2$ arbitrary. Using this, we can define the operator $G:[0,T]\times B\to \gamma(H,E_p)$ as
\[G(t,u)h\coloneqq \Gamma(t,u)Rh,\quad h\in H,\]
where the operator $\Gamma:[0,T]\times B\to \mathcal{L}(H)$ is defined in \eqref{eq:Gammadef}. It is easy to see that $G$ satisfies Assumptions \ref{assum:mainvN}(6) with $\kappa_G=0$. For example, if $u,v\in B$ with $\|u\|,\|v\|\leq r$, then
\begin{align}
\|G(t,u)-G(t,v)\|_{\gamma (H,E_p)}&\leq \|\Gamma(t,u)-\Gamma(t,v)\|_{\mathcal{L}(E_p)}\cdot\|R\|_{\gamma (H,E_p)}\\
&\leq L^{(r)}\cdot\|u-v\|_{B}\cdot\|R\|_{\gamma (H,E_p)}
\end{align}
where $L^{(r)}$ is the maximum of the Lipschitz-constants of the functions $g_j$ on the ball of radius $r$.

If setting \eqref{eq:stochsysmod} instead of  (\ref{eq:stochsys}a), Theorem \ref{theo:SCPnsolcont} remains true as stated; that is, for $q>4$, but one may use a simpler Hilbert space machinery; that is, one may set $p=2$ in the proof. However, in the coloured noise case, Theorem \ref{theo:SCPnsolcont} is true also for $q>2$. But this can only be shown by choosing $p>2$ large enough in the proof and hence, in this case, the Banach space arguments are crucial.

In Theorem \ref{theo:Holderreg}, if one takes $p=2$ (Hilbert space) and $q>4$, then the statement is true for  $\la+\eta>\frac14$ with
\begin{equation}\label{eq:thmHolderreglaetamod}
\la+\eta<\frac12-\frac1q
\end{equation}
instead of \eqref{eq:thmHolderreglaeta}. 
In this case $R$ will be a Hilbert-Schmidt operator whence the covariance operator of the driving process is trace-class. However, the statement of the theorem remains true for $q>2$ as well assuming \eqref{eq:thmHolderreglaetamod} instead of \eqref{eq:thmHolderreglaeta}, but only for the Banach space $E_p$ for $p$ large enough so that $\la+\eta>\frac{1}{2p}$.

The statements of Theorem \ref{thm:SAC} could also be changed accordingly.\vspace{0.5cm}

\section*{Acknowledgements} The authors would like to thank the anonymous referee for her/his useful comments that helped them to improve the presentation of the paper.

M. Kovács acknowledges the support of the Marsden Fund of the Royal Society of New
Zealand through grant no. 18-UOO-143, the Swedish Research Council (VR) through grant
no. 2017-04274 and the NKFIH through grant no. 131545.


\begin{thebibliography}{99}


\bibitem{Al84}
\textsc{F.~Ali~Mehmeti}, Probl\`emes de transmission pour des \'{e}quations des ondes lin\'{e}aires et quasilin\'{e}aires (in French), in: {\em Hyperbolic and holomorphic partial differential equations},
  Travaux en Cours, Hermann, Paris, 1984, pp. 75--96. \MR{747657}

\bibitem{Al94}
\textsc{F. Ali~Mehmeti}, {\em Nonlinear waves in networks}, Mathematical Research, Vol. 80, Akademie-Verlag, Berlin, 1994. \MR{1287844}

\bibitem{AC79}
\textsc{S.~M. Allen, J.~W. Cahn}, A microscopic theory for antiphase boundary motion and its application to antiphase domain coarsening, \textit{Acta Metallurgica}, \textbf{27}(1979), No.~6, 1085--1095.


\bibitem{Ar04}
\textsc{W. Arendt}, Semigroups and evolution equations: functional calculus, regularity and kernel estimates, in: \textit{Evolutionary equations. {V}ol. {I}}, Handb. Differ. Equ., North-Holland, Amsterdam, 2004, pp. 1--85. \MR{2103696}

\bibitem{BFN16}
\textsc{J. Banasiak, A. Falkiewicz, P. Namayanja}, Asymptotic state lumping in transport and diffusion problems on networks with applications to population problems, \textit{Math. Models Methods Appl. Sci.}, \textbf{26}(2016), No.~2, 215--247. \doi{10.1142/S0218202516400017}, \MR{3426200}

\bibitem{BFN16b}
\textsc{J. Banasiak, A. Falkiewicz, P. Namayanja}, Semigroup approach to diffusion and transport problems on networks, \textit{Semigroup Forum}, \textbf{93}(2016), No.~3, 427--443, 2016. \doi{10.1007/s00233-015-9730-4}, \MR{3572410}

\bibitem{BBS92}
\textsc{G.~Barles, L.~Bronsard, and P.~E. Souganidis}, Front propagation for reaction-diffusion equations of bistable type, \textit{Ann. Inst. H. Poincar\'{e} Anal. Non Lin\'{e}aire}, \textbf{9}(1992) No.~5, 479--496. \doi{10.1016/S0294-1449(16)30228-1}, \MR{1191007}

\bibitem{Be85}
\textsc{J. von Below}, A characteristic equation associated to an eigenvalue problem on {$c^2$}-networks, \textit{Linear Algebra Appl.}, \textbf{71}(1985), 309--325. \doi{10.1016/0024-3795(85)90258-7}, \MR{813056}

\bibitem{Be88}
\textsc{J. von Below}, Classical solvability of linear parabolic equations on networks, \textit{J. Differential Equations}, \textbf{72}(1988), No.~2, 316--337. \doi{10.1016/0022-0396(88)90158-1}, \MR{932369}

\bibitem{Be88b}
\textsc{J. von Below}, Sturm-{L}iouville eigenvalue problems on networks, \textit{Math. Methods Appl. Sci.}, \textbf{10}(1988), No.~4, 383--395. \doi{10.1002/mma.1670100404}, \MR{958480}

\bibitem{BN96}
\textsc{J. von Below, S. Nicaise}, Dynamical interface transition in ramified media with diffusion, \textit{Comm. Partial Differential Equations}, \textbf{21}(1996), No.~1-2, 255--279. \doi{10.1080/03605309608821184}, \MR{1373774}

\bibitem{BMZ08}
\textsc{S. Bonaccorsi, C. Marinelli, G. Ziglio}, Stochastic {F}itz{H}ugh-{N}agumo equations on networks with impulsive noise, \textit{Electron. J. Probab.}, \textbf{13}(2008). No.~49, 1362--1379. \doi{10.1214/EJP.v13-532}, \MR{2438810}

\bibitem{BM10}
\textsc{S. Bonaccorsi, D. Mugnolo}, Existence of strong solutions for neuronal network dynamics driven by fractional {B}rownian motions, \textit{Stoch. Dyn.}, \textbf{10}(2010) No.~3, 441--464. \doi{10.1142/S0219493710003030}, \MR{2671386}

\bibitem{BZ14}
\textsc{S. Bonaccorsi, G. Ziglio}, Existence and stability of square-mean almost periodic solutions to spatially extended neural network with impulsive noise, \textit{Random Oper. Stoch. Equ.}, \textbf{22}(2014), No.~1, 17--29. \doi{10.1515/rose-2014-0002}, \MR{3245296}

\bibitem{BG99}
\textsc{Z. Brze\'{z}niak, D. G\c{a}tarek}, Martingale solutions and invariant measures for stochastic evolution equations in {B}anach spaces, \textit{Stochastic Process. Appl.}, \textbf{84}(1999), No.~2, 187--225. \doi{10.1016/S0304-4149(99)00034-4}, \MR{1719282}

\bibitem{BP99}
\textsc{Z. Brze\'{z}niak, S. Peszat}, Space-time continuous solutions to {SPDE}'s driven by a homogeneous {W}iener process, \textit{Studia Math.}, \textbf{137}(1999), No.~3, 261--299. \doi{10.4064/sm-137-3-261-299}, \MR{1736012}

\bibitem{Ca97}
\textsc{C. Cattaneo}, The spectrum of the continuous {L}aplacian on a graph, \textit{Monatsh. Math.}, \textbf{124}(1997), No.~3, 215--235. \doi{10.1007/BF01298245}, \MR{1476363}

\bibitem{CF03}
\textsc{C. Cattaneo, L. Fontana}, D'{A}lembert formula on finite one-dimensional networks, \textit{J. Math. Anal. Appl.}, \textbf{284}(2003), No.~2, 403--424. \doi{10.1016/S0022-247X(02)00392-X}, \MR{1998641}

\bibitem{Ce03}
\textsc{S. Cerrai}, Stochastic reaction-diffusion systems with multiplicative noise and non-{L}ipschitz reaction term, \textit{Probab. Theory Related Fields}, \textbf{125}(2003), No.~2, 271--304. \doi{10.1007/s00440-002-0230-6}, \MR{1961346}

\bibitem{Co70}
\textsc{H.~E. Cook}, Brownian motion in spinodal decomposition, \textit{Acta Metallurgica}, \textbf{18}(1970), No.~3, 297--306.

\bibitem{CP17}
\textsc{F. Cordoni, L. Di~Persio}, Gaussian estimates on networks with dynamic stochastic boundary conditions, \textit{Infin. Dimens. Anal. Quantum Probab. Relat. Top.}, \textbf{20}(2017), No.~1, 1750001, 23. \doi{10.1142/S0219025717500011}, \MR{3623874}

\bibitem{CP17a}
\textsc{F. Cordoni, L. Di~Persio}, Stochastic reaction-diffusion equations on networks with dynamic time-delayed boundary conditions, \textit{J. Math. Anal. Appl.}, \textbf{451}(2017), No.~1, 583--603. \doi{10.1016/j.jmaa.2017.02.008}, \MR{3619253}

\bibitem{DPZ92}
\textsc{G. Da~Prato, J. Zabczyk}, Nonexplosion, boundedness, and ergodicity for stochastic semilinear equations, \textit{J. Differential Equations}, \textbf{98}(1992), No.~1, 181--195. \doi{10.1016/0022-0396(92)90111-Y}, \MR{1168978}

\bibitem{Da90}
\textsc{E.~B. Davies}, \textit{Heat kernels and spectral theory}, Cambridge Tracts in Mathematics, Vol.~92, Cambridge University Press, Cambridge, 1990. \MR{1103113}

\bibitem{EK19}
\textsc{K.-J. Engel, M. Kramar~Fijav\v{z}}, Waves and diffusion on metric graphs with general vertex conditions, \textit{Evol. Equ. Control Theory}, \textbf{8}(2019), No.~3, 633--661. \doi{10.3934/eect.2019030}, \MR{3985968}

\bibitem{EN00}
\textsc{K.-J. Engel, R. Nagel}, \textit{One-parameter semigroups for linear evolution equations}, Graduate Texts in Mathematics, Vol.~194, Springer-Verlag, New York, 2000. \MR{1721989}

\bibitem{Gr87}
\textsc{G. Greiner}, Perturbing the boundary conditions of a generator, \textit{Houston J. Math.}, \textbf{13}(1987), No.~2, 213--229. \MR{904952}

\bibitem{Haase06}
\textsc{M. Haase}, \textit{The functional calculus for sectorial operators}, Operator Theory: Advances and Applications, Vol.~169, Birkh\"{a}user Verlag, Basel, 2006. \doi{10.1007/3-7643-7698-8}, \MR{2244037}

\bibitem{Ka66}
\textsc{M. Kac}, Can one hear the shape of a drum?, \textit{Amer. Math. Monthly}, \textbf{73}(1966), No.~4, 1--23. \doi{10.2307/2313748}, \MR{201237}

\bibitem{KS05}
\textsc{M. Kramar, E. Sikolya}, Spectral properties and asymptotic periodicity of flows in networks, \textit{Math. Z.}, \textbf{249}(2005), No.~1, 139--162. \doi{10.1007/s00209-004-0695-3}, \MR{2106975}

\bibitem{KMS07}
\textsc{M. Kramar~Fijav\v{z}, D. Mugnolo, E. Sikolya}, Variational and semigroup methods for waves and diffusion in networks, \textit{Appl. Math. Optim.}, \textbf{55}(2007), No.~2, 219--240. \doi{10.1007/s00245-006-0887-9}, \MR{2305092}

\bibitem{KvN12}
\textsc{M. Kunze, J. van Neerven}, Continuous dependence on the coefficients and global existence for stochastic reaction diffusion equations, \textit{J. Differential Equations}, \textbf{253}(2012), No.~3, 1036--1068. \doi{10.1016/j.jde.2012.04.013}, \MR{2922662}

\bibitem{KvN19}
\textsc{M. Kunze, J. van Neerven}, Continuous dependence on the coefficients and global existence for stochastic reaction diffusion equations, \url{https://arxiv.org/abs/1104.4258}, 2019.

\bibitem{Lu80}
\textsc{G. Lumer}, Espaces ramifi\'{e}s, et diffusions sur les r\'{e}seaux topologiques (in French), \textit{C. R. Acad. Sci. Paris S\'{e}r. A-B}, \textbf{291}(1980), No.~12, A627--A630. \MR{606449}

\bibitem{MS07}
\textsc{T. M\'{a}trai, E. Sikolya}, Asymptotic behavior of flows in networks, \textit{Forum Math.}, \textbf{19}(2007), No.~3, 429--461. \doi{10.1515/FORUM.2007.018}, \MR{2328116}

\bibitem{Mu07}
\textsc{D. Mugnolo}, Gaussian estimates for a heat equation on a network, \textit{Netw. Heterog. Media}, \textbf{2}(2007), No.~1, 55--79. \doi{10.3934/nhm.2007.2.55}, \MR{2291812}

\bibitem{Mu14}
\textsc{D. Mugnolo}, \textit{Semigroup methods for evolution equations on networks}, Understanding Complex Systems, Springer, Cham, 2014. \doi{10.1007/978-3-319-04621-1}, \MR{3243602}

\bibitem{MR07}
\textsc{D. Mugnolo, S. Romanelli}, Dynamic and generalized {W}entzell node conditions for network equations, \textit{Math. Methods Appl. Sci.}, \textbf{30}(2007), No.~6, 681--706. \doi{10.1002/mma.805}, \MR{2301840}

\bibitem{Ou05}
\textsc{E.~M. Ouhabaz}, \textit{Analysis of heat equations on domains}, London Mathematical Society Monographs Series, Vol.~31, Princeton University Press, Princeton, NJ, 2005. \MR{2124040}

\bibitem{Pe95}
\textsc{S. Peszat}, Existence and uniqueness of the solution for stochastic equations on {B}anach spaces, \textit{Stochastics Stochastics Rep.}, \textbf{55}(1995), No.~3-4, 167--193. \MR{1378855}

\bibitem{RD18}
\textsc{D.~S. Raki\'{c}, D.~S. Djordjevi\'{c}}, A note on topological direct sum of subspaces, \textit{Funct. Anal. Approx. Comput.}, \textbf{10}(2018), No.~1, 9--20. \MR{3782826}

\bibitem{Triebel78}
\textsc{H. Triebel}, \textit{Interpolation theory, function spaces, differential operators}, North-Holland Mathematical Library, Vol.~18, North-Holland Publishing Co., Amsterdam-New York, 1978. \MR{503903}

\bibitem{Tu06}
\textsc{R. Tumulka}, The analogue of {B}ohm-{B}ell processes on a graph, \textit{Phys. Lett. A}, \textbf{348}(2006), No.~3-6, 126--134. \doi{10.1016/j.physleta.2005.08.042}, \MR{2190044}

\bibitem{vNVW08}
\textsc{J.~M. A.~M. van Neerven, M.~C. Veraar, L.~Weis}, Stochastic evolution equations in {UMD} {B}anach spaces, \textit{J. Funct. Anal.}, \textbf{255}(2008), No.~4, 940--993. \doi{10.1016/j.jfa.2008.03.015}, \MR{2433958}



\end{thebibliography}
\end{document}